\documentclass[12pt,reqno]{amsart}
\usepackage{amsmath, amsfonts, amssymb,cite, amsthm, hyperref}
\usepackage{bm}
\usepackage{mathrsfs}
\allowdisplaybreaks[4]
\def\l{\left}
\def\r{\right}
\def\bg{\bigg}
\def\({\bg(}
\def\){\bg)}
\def\t{\text}
\def\f{\frac}

\def\ord{{\rm ord}}

\def\eq{\equiv}

\def\Z{\mathbb Z}
\def\C{\mathbb C}
\def\N{\mathbb N}

\def\<{\langle}
\def\>{\rangle}
\def\1{{\bf 1}}

\theoremstyle{plain}
\newtheorem{theorem}{Theorem}[section]
\newtheorem{conjecture}{Conjecture}

\newtheorem{lemma}{Lemma}[section]

\theoremstyle{definition}

\newtheorem*{Acks}{Acknowledgments}
\theoremstyle{remark}

\numberwithin{equation}{section}

\begin{document}
\title{Supercongruences arising from a ${}_7F_6$ hypergeometric transformation formula}
\author[Chen Wang]{Chen Wang}
\address[Chen Wang]{Department of Applied Mathematics, Nanjing Forestry University, Nanjing 210037, People's Republic of China}
\email{cwang@smail.nju.edu.cn}

\begin{abstract}
Using a ${}_7F_6$ hypergeometric transformation formula, we prove two supercongruences. In particular, one of these supercongruences confirms a recent conjecture of Guo, Liu and Schlosser, and gives an extension of a supercongruence of Long and Ramakrishna.
\end{abstract}


\maketitle

\section{Introduction}

For any $n\in\N=\{0,1,2,\ldots\}$, let $(x)_n=x(x+1)\cdots(x+n-1)$ denote the Pochhammer symbol. For $r\in\N$ and $a_0,\ldots,a_r,b_1,\ldots,b_r,z\in\C$, the hypergeometric series ${}_{r+1}F_r$ are defined as
$$
{}_{r+1}F_r\bigg[\begin{matrix}a_0,&a_1,&\ldots,&a_r\\ &b_1,&\ldots,&b_r\end{matrix}\bigg|\ z\bigg]=\sum_{k=0}^{\infty}\f{(a_0)_k\cdots(a_r)_k}{(b_1)_k\cdots(b_r)_k}\cdot\f{z^k}{k!}.
$$
Partial sums of the hypergeometric series are usually called truncated hypergeometric series which are defined by
$$
{}_{r+1}F_r\bigg[\begin{matrix}a_0,&a_1,&\ldots,&a_r\\ &b_1,&\ldots,&b_r\end{matrix}\bigg|\ z\bigg]_n=\sum_{k=0}^{n}\f{(a_0)_k\cdots(a_r)_k}{(b_1)_k\cdots(b_r)_k}\cdot\f{z^k}{k!},
$$
where $n\in\N$. During the past few decades, supercongruences concerning truncated hypergeometric series have been widely investigated (see e.g., \cite{DFLST16,GuoLiu,GLS,Liu2017,Liu2021,LR,MaoPan2020,MaoPan2022,Mortenson,PTW,Sun2011,WangPan,VanHamme}).

In 1997, Van Hamme \cite{VanHamme} studied the $p$-adic analogues of Ramanujan-type series for $1/\pi$, and conjectured 13 supercongruences concerning truncated hypergeometric series. In particular, Van Hamme's (D.2) supercongruence asserts that for any prime $p\geq5$,
$$
\sum_{k=0}^{p-1}(6k+1)\f{(\f13)_k^6}{(1)_k^6}\eq\begin{cases}-p\Gamma_p\l(\f13\r)^9\pmod{p^4}\quad&\t{if}\ p\eq1\pmod{6},\\ 0\pmod{p^4}\quad&\t{if}\ p\eq5\pmod{6},\end{cases}
$$
where $\Gamma_p$ is the $p$-adic Gamma function (cf. \cite{Morita,Robert00}). In 2016, the (D.2) supercongruence was confirmed by Long and Ramakrishna \cite[Theorem 2]{LR} in the following strengthening form:
\begin{equation}\label{LR}
\sum_{k=0}^{p-1}(6k+1)\f{(\f13)_k^6}{(1)_k^6}\eq\begin{cases}-p\Gamma_p\l(\f13\r)^9\pmod{p^6}\quad&\t{if}\ p\eq1\pmod{6},\\ -\f{10p^4}{27}\Gamma_p\l(\f13\r)^9\pmod{p^6}\quad&\t{if}\ p\eq5\pmod{6}.\end{cases}
\end{equation}
Long and Ramakrishna also pointed out that \eqref{LR} does not hold modulo $p^7$ in general. Later, Guo and Schlosser \cite{GS1,GS} established some partial $q$-analogues of \eqref{LR}. In addition, they obtained a $q$-analogue of the supercongruence
\begin{equation}\label{GSq}
\sum_{k=0}^{p-1}(6k-1)\f{(-\f13)_k^6}{(1)_k^6}\eq0\pmod{p^4},
\end{equation}
where $p\eq1\pmod6$ is a prime. Inspired by \eqref{LR} and \eqref{GSq}, Liu \cite[Theorem 1.2]{Liu2021} generalized \eqref{GSq} as follows:
\begin{equation}\label{Liucon}
\sum_{k=0}^{p-1}(6k-1)\f{(-\f13)_k^6}{(1)_k^6}\eq\begin{cases}140p^4\Gamma_p\l(\f23\r)^9\pmod{p^5}\quad&\t{if}\ p\eq1\pmod{6},\\ 378p\Gamma_p\l(\f23\r)^9\pmod{p^5}\quad&\t{if}\ p\eq5\pmod{6},\end{cases}
\end{equation}
where $p\neq5$ is a prime. It is natural to ask for a parametric extension of \eqref{LR} and \eqref{Liucon}. Recently, Guo, Liu and Schlosser \cite{GLS} gave the common generalization of the second congruence in \eqref{LR}, restricted to modulus $p^5$, and the first congruence in \eqref{Liucon} as follows:
$$
\sum_{k=0}^{p-1}(6k+r)\f{(\f{r}3)_k^6}{(1)_k^6}\eq\ \f{(-1)^{r+1}80rp^4}{81}\cdot\f{\Gamma_p(1+\f{r}3)^2}{\Gamma_p(1+\f{2r}3)^3\Gamma_p(1-\f{r}3)^4}\sum_{k=0}^{1-r}\f{(r-1)_k(\f{r}3)_k^3}{(1)_k(\f{2r}{3})_k^3}\pmod{p^5},
$$
where $r\leq 1$ is an integer coprime with $3$, $p$ is a prime such that $p\eq-r\pmod{3}$ and $p\geq3-r$. They \cite[Conjecture 1]{GLS} also conjectured that the above congruence still holds modulo $p^6$ for $p>3$.

Our first purpose is to confirm \cite[Conjecture 1]{GLS}.

\begin{theorem}\label{GLSconj1}
Let $r\leq 1$ be an integer coprime with $3$. Let $p>3$ be a prime such that $p\eq-r\pmod{3}$ and $p\geq3-r$. Then
\begin{align}\label{GLSconj1eq}
\sum_{k=0}^{p-1}(6k+r)\f{(\f{r}3)_k^6}{(1)_k^6}\eq&\ \f{(-1)^{r+1}80rp^4}{81}\cdot\f{\Gamma_p(1+\f{r}3)^2}{\Gamma_p(1+\f{2r}3)^3\Gamma_p(1-\f{r}3)^4}\sum_{k=0}^{1-r}\f{(r-1)_k(\f{r}3)_k^3}{(1)_k(\f{2r}{3})_k^3}\pmod{p^6}.
\end{align}
\end{theorem}

At the end of their paper, Guo, Liu and Schlosser \cite{GLS} described the obstruction in the proof of Theorem \ref{GLSconj1}. To avoid the obstruction, we utilize the hidden symmetry in the $p$-adic expansion of certain truncated hypergeometric series and $p$-adic Gamma functions.

Guo, Liu and Schlosser \cite[Conjecture 2]{GLS} also conjectured a general extension of the first  congruence in \eqref{LR} and the second congruence in \eqref{Liucon}. However, in the same way as in the proof of Theorem \ref{GLSconj1}, it seems to be very difficult to prove \cite[Conjecture 2]{GLS}.

Our second purpose is to prove the following supercongruence which extends Theorem 1 in \cite{GLS}.

\begin{theorem}\label{GLSth1ex}
Let $r\leq1$ be an odd integer coprime with $5$. Let $p$ be an odd prime such that $p\eq2r\pmod5$ and $p\geq (5-r)/2$. Then
\begin{equation}\label{GLSth1exeq}
\sum_{k=0}^{p-1}(10k+r)\f{(\f{r}{5})_k^5}{(1)_k^5}\eq\f{12p^4}{25}\cdot\f{\Gamma_p(\f r5)^4}{\Gamma_p(\f{2r}5)^2\Gamma_p(\f12+\f{3r}{10})\Gamma_p(\f12-\f r{10})^3}\sum_{k=0}^{(1-r)/2}\f{(\f{r-1}{2})_k(\f r5)_k^3}{(1)_k(\f{2r}5)_k^2(\f12+\f{3r}{10})_k}\pmod{p^5}.
\end{equation}
\end{theorem}

Putting $r=1$ in \eqref{GLSth1ex}, we obtain the following supercongruence:
$$
\sum_{k=0}^{p-1}(10k+1)\f{(\f15)_k^5}{(1)_k^5}\eq\f{12p^4}{25\Gamma_p(\f15)^5\Gamma_p(\f25)^5}\pmod{p^5},
$$
where $p$ is an odd prime with $p\eq2\pmod{5}$.

Guo, Liu and Schlosser \cite[Theorem 1]{GLS} proved the modulus $p^4$ case of \eqref{GLSth1exeq} by using Whipple's well-poised ${}_7F_6$  transformation formula (cf. \cite[Theorem 3.4.5]{AAR99}) and Karlsson-Minton's formula (cf. \cite[Eq. (1.9.2)]{GR}). Differently from them, in our proofs of Theorems \ref{GLSconj1} and \ref{GLSth1ex}, we use only use a ${}_7F_6$ transformation due to Liu \cite[Lemma 2.6]{Liu2021}.

In Sections 2 and 3, we prove Theorems \ref{GLSconj1} and \ref{GLSth1ex}, respectively. In the final section, we provide a conjectural congruence.

\section{Proof of Theorem \ref{GLSconj1}}

Combining Whipple's ${}_7F_6$ transformation formula and a ${}_4F_3$ transformation formula (cf. \cite[Theorem 3.3.3]{AAR99}), Liu \cite[Lemma 2.6]{Liu2021} established the following formula.
\begin{lemma}\label{liu}
Let $n,m$ be nonnegative integers. Then
\begin{align}\label{liueq}
&{}_7F_6\bigg[\begin{matrix}t,&1+\f12t,&-n,&t-a,&t-b,&t-c,&1-t-m+n+a+b+c\\ &\f12t,&1+t+n,&1+a,&1+b,&1+c,&2t+m-n-a-b-c\end{matrix}\bigg|\ 1\bigg]\notag\\
=&\ {}_4F_3\bigg[\begin{matrix}-m,&-n,&a+b+c+1-m-2t,&a+b+c+1+n-m-t\\ &a+b+1-m-t,&a+c+1-m-t,&b+c+1-m-t\end{matrix}\bigg|\ 1\bigg]\notag\\
&\times\f{(1+t)_n(a+b+1-m-t)_n(a+c+1-m-t)_n(b+c+1-m-t)_n}{(1+a)_n(1+b)_n(1+c)_n(a+b+c+1-m-2t)_n}.
\end{align}
\end{lemma}

For any prime $p$, let $\Z_p$ denote the ring of all $p$-adic integers. For $u\in\Z_p$, we use $\<u\>_p$ to denote the last nonnegative residue of $u$ modulo $p$, i.e., $\<u\>_p\in\{0,1,\ldots,p-1\}$ and $u\eq\<u\>_p\pmod{p}$. Our proofs rely on some properties of the $p$-adic Gamma functions.

\begin{lemma}[Robert {\cite[p. 369]{Robert00}}]\label{padicgamma}
Let $p$ be an odd prime. Then, for $x\in\Z_p$, we have
\begin{align*}
&\Gamma_p(0)=1,\quad\Gamma_p(1)=-1,\\
&\Gamma_p(x)\Gamma_p(1-x)=(-1)^{\<-x\>_p-1},\\
&\Gamma_p(x)\eq\Gamma_p(y)\pmod{p}\ \ \t{for}\ \ x\eq y\pmod{p},\\
&\f{\Gamma_p(x+1)}{\Gamma_p(x)}=\begin{cases}-x\ \ \t{if}&\ \ \ord_p(x)=0,\\
-1\ \ \t{if}&\ \ \ord_p(x)>0,
\end{cases}
\end{align*}
where $\ord_p(\cdot)$ stands for the $p$-adic order.
\end{lemma}

For $x\in\Z_p$, set $G_1(x)=\Gamma_p'(x)/\Gamma_p(x)$, where $\Gamma_p'$ stands for the derivative of $\Gamma_p$. Long and Ramakrishna \cite[Theorem 14]{LR} deduced that for $x,\ t\in\Z_p$,
$$
\f{\Gamma_p(x+tp)}{\Gamma_p(x)}\eq 1+G_1(x)tp\pmod{p^2}.
$$
Ahlgren and Ono \cite[Eq. (6.10)]{AhOn00} gave that for any $p$-adic unit $x$,
$$
G_1(x+1)-G_1(x)=\f1x.
$$
It follows that
$$
G_1(x)\eq G_1(0)+\sum_{j=1}^{p-1-\<-x\>_p}\f{1}{j}\pmod{p}.
$$
Combining these, we obtain the following lemma.

\begin{lemma}\label{padicgammaexpansion}
For $x,\ t\in\Z_p$,
$$
\f{\Gamma_p(x+tp)}{\Gamma_p(x)}\eq 1+G_1(0)tp+tp\sum_{j=1}^{p-1-\<-x\>_p}\f{1}{j}\pmod{p^2}.
$$
\end{lemma}

\begin{lemma}\label{uvzeta}
Let $p$ be an odd prime and $\zeta$ be a fifth primitive root of unity. Then, for $u,v\in\Z_p$ and $k\in\{0,1,\ldots,\<-u\>_p\}$,
\begin{equation*}
(u+vp)_k(u+vp\zeta)_k(u+vp\zeta^2)_k(u+vp\zeta^3)_k(u+vp\zeta^4)_k\eq (u)_k^5\l(1+v^5p^5\sum_{j=0}^{k-1}\f{1}{(u+j)^5}\r)\pmod{p^6}.
\end{equation*}
\end{lemma}

\begin{proof}
It is clear that
$$
1+\zeta+\zeta^2+\zeta^3+\zeta^4=\f{1-\zeta^5}{1-\zeta}=0.
$$
Therefore, for $j\in\{0,1,\ldots,k-1\}$, we have
\begin{align*}
&(u+j+vp)(u+j+vp\zeta)(u+j+vp\zeta^2)(u+j+vp\zeta^3)(u+j+vp\zeta^4)\\
=&\  (u+j)^5+(u+j)^4vp(1+\zeta+\zeta^2+\zeta^3+\zeta^4)\\
&+(u+j)^3(vp)^2(\zeta+\zeta^2+2\zeta^3+2\zeta^4+2\zeta^5+\zeta^6+\zeta^7)\\
&+(u+j)^2(vp)^3(\zeta^3+\zeta^4+2\zeta^5+2\zeta^6+2\zeta^7+\zeta^8+\zeta^9)\\
&+(u+j)(vp)^4(\zeta^6+\zeta^7+\zeta^8+\zeta^9+\zeta^{10})+(vp)^5\zeta^{10}\\
=&\ (u+j)^5+(vp)^5,
\end{align*}
and hence,
\begin{align*}
&(u+vp)_k(u+vp\zeta)_k(u+vp\zeta^2)_k(u+vp\zeta^3)_k(u+vp\zeta^4)_k\\
=&\ \prod_{j=0}^{k-1}\l((u+j+vp)(u+j+vp\zeta)(u+j+vp\zeta^2)(u+j+vp\zeta^3)(u+j+vp\zeta^4)\r)\\
=&\ \prod_{j=0}^{k-1}\l((u+j)^5+(vp)^5\r)\\
\eq&\ \prod_{j=0}^{k-1}(u+j)^5\l(1+v^5p^5\sum_{l=0}^{k-1}\f{1}{(u+l)^5}\r)\pmod{p^6}.
\end{align*}
This proves the desired result.
\end{proof}

The following two lemmas concern some hidden symmetry in the $p$-adic expansions of certain truncated hypergeometric series.

\begin{lemma}\label{harmonicsymmetry}
Under the assumptions of Theorem \ref{GLSconj1}, we have
$$\sum_{k=0}^{(2p-r)/3}(6k+r)\f{(\f r3)_k^6}{(1)_k^6}\l(\sum_{j=0}^{k-1}\f{1}{(r/3+j)^5}+\sum_{j=1}^k\f{1}{j^5}\r)\eq0\pmod{p}.$$
\end{lemma}
\begin{proof}
Clearly,
\begin{align*}
&\sum_{k=0}^{(2p-r)/3}(6k+r)\f{(\f r3)_k^6}{(1)_k^6}\sum_{j=1}^k\f{1}{j^5}\\
=&\sum_{k=0}^{(2p-r)/3}\l(6\l(\f{2p-r}{3}-k\r)+r\r)\f{(\f r3)_{(2p-r)/3-k}^6}{(1)_{(2p-r)/3-k}^6}\sum_{j=1}^{(2p-r)/3-k}\f1{j^5}\\
=&\ \f{(\f r3)_{(2p-r)/3}^6}{(1)_{(2p-r)/3}^6}\sum_{k=0}^{(2p-r)/3}(4p-6k-r)\f{(\f{r}{3}-\f{2p}{3})_k^6}{(1-\f{2p}3)_k^6}\sum_{j=k}^{(2p-r)/3-1}\f1{((2p-r)/3-j)^5}\\
\eq&\sum_{k=0}^{(2p-r)/3}(6k+r)\f{(\f r3)_k^6}{(1)_k^6}\sum_{j=k}^{(2p-r)/3-1}\f{1}{(r/3+j)^5}\pmod{p},
\end{align*}
where we have used
$$
\f{(\f r3)_{(2p-r)/3}^6}{(1)_{(2p-r)/3}^6}=\f{\prod_{j=0}^{(2p-r)/3-1}(\f r3+j)^6}{\prod_{j=0}^{(2p-r)/3-1}(1+j)^6}=\f{\prod_{j=0}^{(2p-r)/3-1}(\f r3+\f{2p-r}{3}-1-j)^6}{\prod_{j=0}^{(2p-r)/3-1}(1+j)^6}\eq1\pmod{p}.
$$
Thus we obtain
\begin{align*}
&\sum_{k=0}^{(2p-r)/3}(6k+r)\f{(\f r3)_k^6}{(1)_k^6}\l(\sum_{j=0}^{k-1}\f{1}{(r/3+j)^5}+\sum_{j=1}^k\f{1}{j^5}\r)\\
\eq&\sum_{k=0}^{(2p-r)/3}(6k+r)\f{(\f r3)_k^6}{(1)_k^6}\sum_{j=0}^{(2p-r)/3-1}\f{1}{(r/3+j)^5}\pmod{p}.
\end{align*}
Then the desired result follows from the fact (cf. \cite[Theorem 2]{GLS})
$$
\sum_{k=0}^{(2p-r)/3}(6k+r)\f{(\f r3)_k^6}{(1)_k^6}\eq0\pmod{p}.
$$
\end{proof}

\begin{lemma}\label{4F3symmetry}
Under the assumptions of Theorem \ref{GLSconj1}, we have
\begin{equation}\label{4F3symmetryeq}
\sum_{k=0}^{1-r}\f{(r-1)_k(\f r3)_k^3}{(1)_k(\f{2r}3)_k^3}\l(\sum_{j=0}^{k-1}\f{1}{\f{2r}3+j}-\sum_{j=0}^{k-1}\f{1}{\f{r}3+j}\r)=\sum_{k=0}^{1-r}\f{(r-1)_k(\f r3)_k^3}{(1)_k(\f{2r}3)_k^3}\sum_{j=0}^{-r}\f{1}{\f{2r}3+j}.
\end{equation}
\end{lemma}

\begin{proof}
Note that
$$
(r-1)_{1-r}=(-1)^{1-r}(1)_{1-r}\quad \t{and}\quad \l(\f r3\r)_{1-r}=(-1)^{1-r}\l(\f{2r}3\r)_{1-r}.
$$
Thus we have
\begin{align*}
\sum_{k=0}^{1-r}\f{(r-1)_k(\f r3)_k^3}{(1)_k(\f{2r}3)_k^3}\sum_{j=0}^{k-1}\f{1}{\f{r}3+j}=&\sum_{k=0}^{1-r}\f{(r-1)_{1-r-k}(\f r3)_{1-r-k}^3}{(1)_{1-r-k}(\f{2r}3)_{1-r-k}^3}\sum_{j=0}^{-r-k}\f{1}{\f{r}3+j}\\
=&\ \f{(r-1)_{1-r}(\f r3)_{1-r}^3}{(1)_{1-r}(\f{2r}3)_{1-r}^3}\sum_{k=0}^{1-r}\f{(r-1)_k(\f r3)_k^3}{(1)_k(\f{2r}3)_k^3}\sum_{j=k}^{-r}\f{1}{\f{r}3-r-j}\\
=&\ -\sum_{k=0}^{1-r}\f{(r-1)_k(\f r3)_k^3}{(1)_k(\f{2r}3)_k^3}\sum_{j=k}^{-r}\f{1}{\f{2r}3+j}.
\end{align*}
Substituting this into the left-hand side of \eqref{4F3symmetryeq}, we obtain the desired result.
\end{proof}

\begin{lemma}\label{harmonicsymmetry2}
Under the assumptions of Theorem \ref{GLSconj1}, we have
$$
\sum_{j=0}^{(p-2r-3)/3}\f{1}{\f{2r}3+j}+\sum_{j=1}^{(p+r-3)/3}\f1j-\sum_{j=0}^{-r}\f{1}{\f{2r}3+j}\eq0\pmod{p}.
$$
\end{lemma}

\begin{proof}
Set $d=(r+p)/3$. Clearly,
\begin{align*}
&\sum_{j=0}^{(p-2r-3)/3}\f{1}{\f{2r}3+j}+\sum_{j=1}^{(p+r-3)/3}\f1j-\sum_{j=0}^{-r}\f{1}{\f{2r}3+j}\\
\eq& \sum_{j=0}^{p-2d-1}\f{1}{2d+j}+\sum_{j=1}^{d-1}\f1j-\sum_{j=0}^{p-3d}\f{1}{2d+j}\\
=&\ \sum_{j=2d}^{p-1}\f1j-\sum_{j=2d}^{p-d}\f1j+\sum_{j=1}^{d-1}\f1j\\
=&\ \sum_{j=p-d+1}^{p-1}\f1j+\sum_{j=1}^{d-1}\f1j\\
=&\ \sum_{j=1}^{d-1}\f{1}{p-j}+\sum_{j=1}^{d-1}\f1j\\
\eq&\ 0\pmod{p}.
\end{align*}
\end{proof}

\noindent{\it Proof of Theorem \ref{GLSconj1}}. Let $\zeta$ be a fifth primitive root of unity. Putting $m=1-r,\ t=\f r3,\ n = \f{2p-r}3,\ a=\f{2p\zeta}3,\ b=\f{2p\zeta^2}3,\ c= \f{2p\zeta^3}3$ in \eqref{liu} and with helps of Lemmas \ref{uvzeta} and \ref{harmonicsymmetry}, the left-hand side of \eqref{liueq} becomes
\begin{align*}
&{}_7F_6\bigg[\begin{matrix}1+\f r6,&\f r3,&\f{r-2p}3,&\f{r-2p\zeta}3,&\f{r-2p\zeta^2}3,&\f{r-2p\zeta^3}3,&\f{r-2p\zeta^4}3\\ &\f r6,&1+\f{2p}3,&1+\f{2p\zeta}3,&1+\f{2p\zeta^2}3,&1+\f{2p\zeta^3}3,&1+\f{2p\zeta^4}3\end{matrix}\bigg|\ 1\bigg]\\
=&\ {}_7F_6\bigg[\begin{matrix}1+\f r6,&\f r3,&\f{r-2p}3,&\f{r-2p\zeta}3,&\f{r-2p\zeta^2}3,&\f{r-2p\zeta^3}3,&\f{r-2p\zeta^4}3\\ &\f r6,&1+\f{2p}3,&1+\f{2p\zeta}3,&1+\f{2p\zeta^2}3,&1+\f{2p\zeta^3}3,&1+\f{2p\zeta^4}3\end{matrix}\bigg|\ 1\bigg]_{\f{2p-r}{3}}\\
\eq&\ \f1r\sum_{k=0}^{(2p-r)/3}(6k+r)\f{(\f r3)_k^6}{(1)_k^6}\l(1-\f{32}{243}p^5\sum_{j=0}^{k-1}\f{1}{(r/3+j)^5}-\f{32}{243}p^5\sum_{j=1}^k\f1{j^5}\r)\\
\eq&\ \f1r\sum_{k=0}^{(2p-r)/3}(6k+r)\f{(\f r3)_k^6}{(1)_k^6}\\
\eq&\ \f1r\sum_{k=0}^{p-1}(6k+r)\f{(\f r3)_k^6}{(1)_k^6}\pmod{p^6},
\end{align*}
where in the last step we have used the fact that $(\f r3)_k\eq0\pmod{p}$ for $k\in\{(2p-r)/3+1,(2p-r)/3+2,\ldots,p-1\}$.

On the other hand, by Lemma \ref{liu},
\begin{align*}
&{}_7F_6\bigg[\begin{matrix}1+\f r6,&\f r3,&\f{r-2p}3,&\f{r-2p\zeta}3,&\f{r-2p\zeta^2}3,&\f{r-2p\zeta^3}3,&\f{r-2p\zeta^4}3\\ &\f r6,&1+\f{2p}3,&1+\f{2p\zeta}3,&1+\f{2p\zeta^2}3,&1+\f{2p\zeta^3}3,&1+\f{2p\zeta^4}3\end{matrix}\bigg|\ 1\bigg]\\
=&\ {}_4F_3\bigg[\begin{matrix}r-1,&\f{r-2p}{3},&\f r3-\f{2p}{3}(1+\zeta^4),&\f r3-\f{2p\zeta^4}{3}\\ &\f{2r}3+\f{2p}{3}(\zeta+\zeta^2),&\f{2r}3+\f{2p}{3}(\zeta+\zeta^3),&\f{2r}3+\f{2p}{3}(\zeta^2+\zeta^3)\end{matrix}\bigg|\ 1\bigg]\\
&\times\f{(1+\f r3)_{(2p-r)/3}(\f{2r}3+\f{2p}{3}(\zeta+\zeta^2))_{(2p-r)/3}(\f{2r}3+\f{2p}{3}(\zeta+\zeta^3))_{(2p-r)/3}(\f{2r}3+\f{2p}{3}(\zeta^2+\zeta^3))_{(2p-r)/3}}{(1+\f{2p\zeta}{3})_{(2p-r)/3}(1+\f{2p\zeta^2}{3})_{(2p-r)/3}(1+\f{2p\zeta^3}{3})_{(2p-r)/3}(\f r3+\f{2p}3(\zeta+\zeta^2+\zeta^3))_{(2p-r)/3}}.
\end{align*}
Note that
$$
\l(1+\f r3\r)_{\f{2p-r}{3}}=\f{2p}{3}\l(1+\f r3\r)_{\f{2p-r-3}{3}}=\f{2p}{r}\l(\f r3\r)_{\f{2p-r}{3}}
$$
and
\begin{align*}
&\l(\f{2r}{3}+\f{2p}{3}(\zeta^i+\zeta^j)\r)_{\f{2p-r}{3}}\\
=&\ \f p3(1+2\zeta^i+2\zeta^j)\l(\f{2r}3+\f{2p}{3}(\zeta^i+\zeta^j)\r)_{\f{p-2r}{3}}\l(1+\f p3(1+2\zeta^i+2\zeta^j)\r)_{\f{p+r-3}{3}}.
\end{align*}
Hence,
\begin{align*}
&\l(1+\f r3\r)_{\f{2p-r}{3}}\l(\f{2r}{3}+\f{2p}{3}(\zeta+\zeta^2)\r)_{\f{2p-r}{3}}\l(\f{2r}{3}+\f{2p}{3}(\zeta+\zeta^3)\r)_{\f{2p-r}{3}}\l(\f{2r}{3}+\f{2p}{3}(\zeta^2+\zeta^3)\r)_{\f{2p-r}{3}}\\
=&\ \f{10p^4}{27r}\l(\f r3\r)_{\f{2p-r}{3}}\l(\f{2r}3+\f{2p}{3}(\zeta+\zeta^2)\r)_{\f{p-2r}{3}}\l(1+\f p3(1+2\zeta+2\zeta^2)\r)_{\f{p+r-3}{3}}\l(\f{2r}3+\f{2p}{3}(\zeta+\zeta^3)\r)_{\f{p-2r}{3}}\\
&\times\l(1+\f p3(1+2\zeta+2\zeta^3)\r)_{\f{p+r-3}{3}}\l(\f{2r}3+\f{2p}{3}(\zeta^2+\zeta^3)\r)_{\f{p-2r}{3}}\l(1+\f p3(1+2\zeta^2+2\zeta^3)\r)_{\f{p+r-3}{3}}\\
\eq&\ \f{10p^4}{27r}\l(\f r3\r)_{\f{2p-r}{3}}\l(\f {2r}3\r)_{\f{p-2r}{3}}^3\big(1\big)_{\f{p+r-3}{3}}^3\bigg(1+\f{2p}{3}(2\zeta+2\zeta^2+2\zeta^3)\sum_{j=0}^{(p-2r-3)/3}\f{1}{\f{2r}3+j}\\
&+\f{p}{3}(3+4\zeta+4\zeta^2+4\zeta^3)\sum_{j=1}^{(p+r-3)/3}\f1j\bigg)\\
=&\ \f{10p^4}{27r}\l(\f r3\r)_{\f{2p-r}{3}}\l(\f {2r}3\r)_{\f{p-2r}{3}}^3\big(1\big)_{\f{p+r-3}{3}}^3\Bigg(1-\f{4p}{3}(1+\zeta^4)\bigg(\sum_{j=0}^{(p-2r-3)/3}\f{1}{\f{2r}3+j}+\sum_{j=1}^{(p+r-3)/3}\f1j\bigg)\\
&+p\sum_{j=1}^{(p+r-3)/3}\f1j\Bigg)\pmod{p^6}.
\end{align*}
Moreover, by Lemma \ref{uvzeta},
\begin{align*}
&\l(1+\f{2p\zeta}{3}\r)_{\f{2p-r}3}\l(1+\f{2p\zeta^2}{3}\r)_{\f{2p-r}{3}}\l(1+\f{2p\zeta^3}{3}\r)_{\f{2p-r}3}\l(\f r3+\f{2p}3(\zeta+\zeta^2+\zeta^3)\r)_{\f{2p-r}3}\\
=&\  \f{(-1)^{\f{2p-r}{3}}\l(1+\f{2p}{3}\r)_{\f{2p-r}3}\l(1+\f{2p\zeta}{3}\r)_{\f{2p-r}3}\l(1+\f{2p\zeta^2}{3}\r)_{\f{2p-r}{3}}\l(1+\f{2p\zeta^3}{3}\r)_{\f{2p-r}3}\l(1+\f{2p\zeta^4}3\r)_{\f{2p-r}3}}{\l(1+\f{2p}{3}\r)_{\f{2p-r}3}}\\
\eq&\ \f{(-1)^{\f{2p-r}{3}}(1)_{\f{2p-r}3}^5}{\l(1+\f{2p}{3}\r)_{\f{2p-r}3}}\pmod{p^2}.
\end{align*}
Meanwhile, with the help of Lemma \ref{4F3symmetry}, we have
\begin{align*}
&{}_4F_3\bigg[\begin{matrix}r-1,&\f{r-2p}{3},&\f r3-\f{2p}{3}(1+\zeta^4),&\f r3-\f{2p\zeta^4}{3}\\ &\f{2r}3+\f{2p}{3}(\zeta+\zeta^2),&\f{2r}3+\f{2p}{3}(\zeta+\zeta^3),&\f{2r}3+\f{2p}{3}(\zeta^2+\zeta^3)\end{matrix}\bigg|\ 1\bigg]\\
\eq&\ \sum_{k=0}^{1-r}\f{(r-1)_k(\f r3)_k^3}{(1)_k(\f{2r}3)_k^3}\l(1+\f{4p}3(1+\zeta^4)\l(\sum_{j=0}^{k-1}\f{1}{\f{2r}3+j}-\sum_{j=0}^{k-1}\f{1}{\f{r}3+j}\r)\r)\\
=&\ \sum_{k=0}^{1-r}\f{(r-1)_k(\f r3)_k^3}{(1)_k(\f{2r}3)_k^3}\l(1+\f{4p}3(1+\zeta^4)\sum_{j=0}^{-r}\f{1}{\f{2r}3+j}\r)\pmod{p^2}.
\end{align*}

In view of the above and Lemma \ref{harmonicsymmetry2}, it suffices to show
\begin{align}\label{GLSconj1key}
&\f{(-1)^{\f{2p-r}{3}}(\f r3)_{\f{2p-r}{3}}(\f{2r}3)_{\f{p-2r}{3}}^3(1)_{\f{p+r-3}{3}}^3\l(1+\f{2p}{3}\r)_{\f{2p-r}3}}{(1)_{\f{2p-r}3}^5}\l(1+p\sum_{j=1}^{(p+r-3)/3}\f1j\r)\notag\\
\eq&\ \f{(-1)^{r+1}8r}{3}\cdot\f{\Gamma_p(1+\f{r}3)^2}{\Gamma_p(1+\f{2r}3)^3\Gamma_p(1-\f{r}3)^4}\pmod{p^2}.
\end{align}
Observe that
$$
(-1)^{\f{2p-r}{3}}\l(\f r3\r)_{\f{2p-r}{3}}\l(1+\f{2p}{3}\r)_{\f{2p-r}3}=\l(1-\f{2p}{3}\r)_{\f{2p-r}3}\l(1+\f{2p}{3}\r)_{\f{2p-r}3}\eq\big(1\big)_{\f{2p-r}3}^2\pmod{p^2}
$$
and
\begin{align*}
&\f{\l(\f{2r}{3}\r)_{\f{p-2r}{3}}^3\big(1\big)_{\f{p+r-3}{3}}^3}{\big(1\big)_{\f{2p-r}3}^3}=\f{\Gamma_p(\f{p}3)^3\Gamma_p(\f{r}3+\f{p}3)^3}{\Gamma_p(\f{2r}3)^3\Gamma_p(1-\f{r}3+\f{2p}3)^3}\\
\eq&\f{(-1)^{r+1}8r}{3}\cdot\f{\Gamma_p(1+\f{r}3)^2}{\Gamma_p(1+\f{2r}3)^3\Gamma_p(1-\f{r}3)^4}\l(1+p\sum_{j=1}^{(2p-r)/3}\f1j-2p\sum_{j=1}^{(p+r-3)/3}\f1j\r)\pmod{p^2},
\end{align*}
where in the last congruence we have used Lemmas \ref{padicgamma} and \ref{padicgammaexpansion}. It is known that $\sum_{j=1}^{p-1}\f1j\eq0\pmod{p}$. Thus
$$
\sum_{j=1}^{(2p-r)/3}\f1j=\sum_{j=(p+r)}^{p-1}\f{1}{p-j}\eq-\l(\sum_{j=1}^{p-1}\f1j-\sum_{j=1}^{p+r-3}\f1j\r)\eq\sum_{j=1}^{p+r-3}\f1j\pmod{p}.
$$
Combining these, we have
\begin{align*}
&\f{(-1)^{\f{2p-r}{3}}(\f r3)_{\f{2p-r}{3}}(\f{2r}3)_{\f{p-2r}{3}}^3(1)_{\f{p+r-3}{3}}^3\l(1+\f{2p}{3}\r)_{\f{2p-r}3}}{(1)_{\f{2p-r}3}^5}\l(1+p\sum_{j=1}^{(p+r-3)/3}\f1j\r)\\
\eq&\f{(-1)^{r+1}8r}{3}\cdot\f{\Gamma_p(1+\f{r}3)^2}{\Gamma_p(1+\f{2r}3)^3\Gamma_p(1-\f{r}3)^4}\l(1-p\sum_{j=1}^{(p+r-3)/3}\f1j\r)\l(1+p\sum_{j=1}^{(p+r-3)/3}\f1j\r)\\
\eq&\f{(-1)^{r+1}8r}{3}\cdot\f{\Gamma_p(1+\f{r}3)^2}{\Gamma_p(1+\f{2r}3)^3\Gamma_p(1-\f{r}3)^4}\pmod{p^2}.
\end{align*}
This proves \eqref{GLSconj1key}.

The proof is now complete.\qed

\medskip

\section{Proof of Theorem \ref{GLSth1ex}}
Throughout this section, we use $i$ to denote the imaginary unit.
\begin{lemma}\label{th2lem1}
Let $p$ be an odd prime. Then, for $u,v\in\Z_p$ and $k\in\{0,1,\ldots,\<-u\>_p\}$,
\begin{equation*}
(u+vp)_k(u-vp)_k(u+vip)_k(u-vip)_k\eq(u)_k^4\l(1+v^4p^4\sum_{j=0}^{k-1}\f{1}{(u+j)^4}\r)\pmod{p^5}.
\end{equation*}
\end{lemma}

\begin{proof}
We omit the proof, since it is quite similar to the one of Lemma \ref{uvzeta}.
\end{proof}

\begin{lemma}\label{th2lem2}
Under the assumptions of Theorem \ref{GLSconj1}, we have
$$
\sum_{k=0}^{(3p-r)/5}(10k+r)\f{(\f r5)_k^5}{(1)_k^5}\l(\sum_{j=0}^{k-1}\f{1}{(r/5+j)^4}-\sum_{j=1}^k\f1{j^4}\r)\eq0\pmod{p}.
$$
\end{lemma}

\begin{proof}
Obviously,
\begin{align*}
&\sum_{k=0}^{(3p-r)/5}(10k+r)\f{(\f r5)_k^5}{(1)_k^5}\sum_{j=1}^k\f1{j^4}\\
=&\sum_{k=0}^{(3p-r)/5}\l(10\l(\f{3p-r}{5}-k\r)+r\r)\f{(\f r5)_{(3p-r)/5-k}^5}{(1)_{(3p-r)/5-k}^5}\sum_{j=1}^{(3p-r)/5-k}\f1{j^4}\\
=&\f{(\f r5)_{(3p-r)/5}^5}{(1)_{(3p-r)/5}^5}\sum_{k=0}^{(3p-r)/5}(6p-r-10k)\f{(\f r5-\f{3p}5)_k^5}{(1-\f{3p}5)_k^5}\sum_{j=k}^{(3p-r)/5-1}\f{1}{((3p-r)/5-j)^4}\\
\eq&-\sum_{k=0}^{(3p-r)/5}(10k+r)\f{(\f r5)_k^5}{(1)_k^5}\sum_{j=k}^{(3p-r)/5-1}\f{1}{(r/5+j)^4}\pmod{p},
\end{align*}
where in the last step we used
\begin{align*}
\f{(\f r5)_{(3p-r)/5}^5}{(1)_{(3p-r)/5}^5}=&\f{\prod_{j=0}^{(3p-r)/5-1}(\f r5+j)^5}{\prod_{j=0}^{(3p-r)/5-1}(1+j)^5}=\f{\prod_{j=0}^{(3p-r)/5-1}(\f r5+\f{3p-r}{5}-1-j)^5}{\prod_{j=0}^{(3p-r)/5-1}(1+j)^5}\\
\eq&(-1)^{(3p-r)/5}=1\pmod{p}.
\end{align*}
Therefore,
\begin{align*}
&\sum_{k=0}^{(3p-r)/5}(10k+r)\f{(\f r5)_k^5}{(1)_k^5}\l(\sum_{j=0}^{k-1}\f{1}{(r/5+j)^4}+\sum_{j=1}^k\f1{j^4}\r)\\
\eq&-\sum_{k=0}^{(3p-r)/5}(10k+r)\f{(\f r5)_k^5}{(1)_k^5}\sum_{j=0}^{(3p-r)/5-1}\f{1}{(r/5+j)^4}\pmod{p}.
\end{align*}
Then the desired lemma follows from the fact (cf. \cite[Theorem 1]{GLS})
$$
\sum_{k=0}^{(3p-r)/5}(10k+r)\f{(\f r5)_k^5}{(1)_k^5}\eq0\pmod{p}.
$$
\end{proof}

\medskip

\noindent{\it Proof of Theorem \ref{GLSth1ex}}. We can directly verify \eqref{GLSth1exeq} for $p=2$ and $r=1$. Below we assume that $p$ is an odd prime. Putting $m=\f{1-r}{2},\ t=\f r5,\ n=\f{3p-r}{5},\ a=\f{r}{10}-\f12,\ b=-\f{3p}{5},\ c=-\f{3ip}{5}$ in \eqref{liueq}, and with helps of Lemmas \ref{th2lem1} and \ref{th2lem2}, the left-hand side of \eqref{liueq} becomes
\begin{align*}
&{}_7F_6\bigg[\begin{matrix}\f r5,&1+\f r{10},&\f{r}{5}-\f{3p}{5},&\f{1}{2}+\f{r}{10},&\f{r}{5}+\f{3p}{5},&\f{r}{5}+\f{3ip}{5},&\f{r}{5}-\f{3ip}{5}\\ &\f{r}{10},&1+\f{3p}{5},&\f{1}{2}+\f{r}{10},&1-\f{3p}{5},&1-\f{3ip}{5},&1+\f{3ip}{5}\end{matrix}\bigg|\ 1\bigg]\\
=&{}_6F_5\bigg[\begin{matrix}\f r5,&1+\f r{10},&\f{r}{5}-\f{3p}{5},&\f{r}{5}+\f{3p}{5},&\f{r}{5}+\f{3ip}{5},&\f{r}{5}-\f{3ip}{5}\\ &\f{r}{10},&1+\f{3p}{5},&1-\f{3p}{5},&1-\f{3ip}{5},&1+\f{3ip}{5}\end{matrix}\bigg|\ 1\bigg]_{\f{3p-r}{5}}\\
\eq&\f1r\sum_{k=0}^{(3p-r)/5}(10k+r)\f{(\f r5)_k^5}{(1)_k^5}\l(1-\f{81}{625}p^4\sum_{j=0}^{k-1}\f{1}{(r/5+j)^4}+\f{81}{625}p^4\sum_{j=1}^k\f{1}{j^4}\r)\\
\eq&\f1r\sum_{k=0}^{(3p-r)/5}(10k+r)\f{(\f r5)_k^5}{(1)_k^5}\\
\eq&\f1r\sum_{k=0}^{p-1}(10k+r)\f{(\f r5)_k^5}{(1)_k^5},
\end{align*}
where in the last step we noted that $(\f r5)_k\eq0\pmod{p}$ for $k$ among $(3p-r)/5+1,(3p-r)/5+2,\ldots,p-1$.

On the other hand, in view of Lemma \ref{liu}, we have
\begin{align}\label{4F3}
&{}_7F_6\bigg[\begin{matrix}\f r5,&1+\f r{10},&\f{r}{5}-\f{3p}{5},&\f{1}{2}+\f{r}{10},&\f{r}{5}+\f{3p}{5},&\f{r}{5}+\f{3ip}{5},&\f{r}{5}-\f{3ip}{5}\\ &\f{r}{10},&1+\f{3p}{5},&\f{1}{2}+\f{r}{10},&1-\f{3p}{5},&1-\f{3ip}{5},&1+\f{3ip}{5}\end{matrix}\bigg|\ 1\bigg]\notag\\
=&{}_4F_3\bigg[\begin{matrix}\f{r-1}{2},&\f{r}{5}-\f{3p}{5},&\f{r}{5}-\f{3(1+i)p}{5},&\f{r}{5}-\f{3(-1+i)p}{5}\\ &\f{2r}{5}-\f{3p}{5},&\f{2r}{5}-\f{3ip}{5},&\f{1}{2}+\f{3r}{10}-\f{3(1+i)p}{5}\end{matrix}\bigg|\ 1\bigg]\notag\\
&\times\f{(1+\f r5)_{(3p-r)/5}(\f{2r}{5}-\f{3p}{5})_{(3p-r)/5}(\f{2r}{5}-\f{3ip}{5})_{(3p-r)/5}(\f{1}{2}+\f{3r}{10}-\f{3(1+i)p}{5})_{(3p-r)/5}}{(\f12+\f{r}{10})_{(3p-r)/5}(1-\f{3p}{5})_{(3p-r)/5}(1-\f{3ip}{5})_{(3p-r)/5}(\f{r}{5}-\f{3(1+i)p}{5})_{(3p-r)/5}}.
\end{align}
Since $p\geq (5-r)/2$,
$$
\f{3p}{5}-\f{r}{5}-\f{1-r}{2}=\f{6p+3r-5}{10}\geq1.
$$
Therefore, the ${}_4F_3$ series on the right-hand side of \eqref{4F3} terminates at the $(\f{1-r}{2})$-th term. For any $k\in\{0,1,\ldots,(1-r)/2\}$, all rising factorials in the summands of the ${}_4F_3$ series are not divisible by $p$. Moreover, it is obvious that
\begin{align*}
\l(1+\f{r}{5}\r)_{\f{3p-r}{5}}&=\l(1+\f{r}{5}\r)\l(2+\f{r}{5}\r)\cdots\f{3p}{5},\\
\l(\f{2r}{5}-\f{3p}{5}\r)_{\f{3p-r}{5}}&=\l(\f{2r}{5}-\f{3p}{5}\r)\l(1+\f{2r}{5}-\f{3p}{5}\r)\cdots\l(-\f{2p}{5}\r)\cdots\l(\f{r}{5}-1\r),\\
\l(\f{2r}{5}-\f{3ip}{5}\r)_{\f{3p-r}{5}}&=\l(\f{2r}{5}-\f{3ip}{5}\r)\l(1+\f{2r}{5}-\f{3ip}{5}\r)\cdots\f{(1-3i)p}{5}\cdots\l(\f{r}{5}+\f{3(1-i)p}{5}-1\r),\\
\l(\f{1}{2}+\f{3r}{10}-\f{3(1+i)p}{5}\r)_{\f{3p-r}{5}}&=\l(\f{1}{2}+\f{3r}{10}-\f{3(1+i)p}{5}\r)\l(\f{3}{2}+\f{3r}{10}-\f{3(1+i)p}{5}\r)\cdots\f{(-1-3i)p}{5}\\
&\quad\ \cdots\l(-\f12+\f{r}{10}-\f{3ip}{5}\r),
\end{align*}
and
$$
\l(\f12+\f{r}{10}\r)_{\f{3p-r}{5}}\l(1-\f{3p}{5}\r)_{\f{3p-r}{5}}\l(1-\f{3ip}{5}\r)_{\f{3p-r}{5}}\l(\f{r}{5}-\f{3(1+i)p}{5}\r)_{\f{3p-r}{5}}\not\eq0\pmod{p}.
$$
Hence, modulo $p^5$, the right-hand sides of \eqref{4F3} becomes
\begin{align*}
&{}_4F_3\bigg[\begin{matrix}\f{r-1}{2},&\f{r}{5},&\f{r}{5},&\f{r}{5}\\ &\f{2r}{5},&\f{2r}{5},&\f{1}{2}+\f{3r}{10}\end{matrix}\bigg|\ 1\bigg]\\
&\qquad\times \f{(1+\f r5)_{(3p-r)/5}(\f{2r}{5}-\f{3p}{5})_{(3p-r)/5}(\f{2r}{5}-\f{3ip}{5})_{(3p-r)/5}(\f{1}{2}+\f{3r}{10}-\f{3(1+i)p}{5})_{(3p-r)/5}}{(\f12+\f{r}{10})_{(3p-r)/5}(1)_{(3p-r)/5}^2(\f{r}{5})_{(3p-r)/5}}.
\end{align*}
Using Lemma \ref{padicgamma} and writing the quotients of rising factorials in the above expression in terms of $p$-adic Gamma quotients, we obtain
\begin{align*}
&\f{(1+\f r5)_{(3p-r)/5}(\f{2r}{5}-\f{3p}{5})_{(3p-r)/5}(\f{2r}{5}-\f{3ip}{5})_{(3p-r)/5}(\f{1}{2}+\f{3r}{10}-\f{3(1+i)p}{5})_{(3p-r)/5}}{(\f12+\f{r}{10})_{(3p-r)/5}(1)_{(3p-r)/5}^2(\f{r}{5})_{(3p-r)/5}}\\
\eq&\f{12p^4}{25r}\cdot\f{\Gamma_p(\f r5)^2\Gamma_p(\f12+\f{r}{10})^2\Gamma_p(1)^2}{\Gamma_p(\f{2r}{5})^2\Gamma_p(\f12+\f{3r}{10})\Gamma_p(\f12-\f{r}{10})\Gamma_p(1-\f r5)^2}\\
=&\f{12p^4}{25r}\cdot\f{\Gamma_p(\f r5)^4}{\Gamma_p(\f{2r}{5})^2\Gamma_p(\f12+\f{3r}{10})\Gamma_p(\f12-\f{r}{10})^3}\pmod{p^5}.
\end{align*}

Combining the above, we finish the proof.\qed

\section{An open conjecture}

Motivated by Theorem \ref{GLSth1ex} and based on some numerical evidence, we propose the following conjecture.

\begin{conjecture}\label{conj1}
Let $r\leq 1$ be an odd integer coprime with $5$. Let $p$ be and odd prime such that $p\eq r\pmod{5}$ and $p\geq (5-3r)/2$. Then
\begin{equation}\label{conj1eq}
\sum_{k=0}^{p-1}(10k+r)\f{(\f{r}{5})_k^5}{(1)_k^5}\eq\f{p\Gamma_p(\f r5)^4}{\Gamma_p(\f{2r}5)^2\Gamma_p(\f12+\f{3r}{10})\Gamma_p(\f12-\f r{10})^3}\sum_{k=0}^{(1-r)/2}\f{(\f{r-1}{2})_k(\f r5)_k^3}{(1)_k(\f{2r}5)_k^2(\f12+\f{3r}{10})_k}\pmod{p^5}.
\end{equation}
\end{conjecture}

In a similar way as in the proof of Theorem \ref{GLSth1ex}, we can only prove \eqref{conj1eq} in the modulus $p^2$ case. We hope that an interested reader will make some progress on it.

\begin{Acks}
This work is supported by the National Natural Science Foundation of China (grant no. 12201301).
\end{Acks}


\begin{thebibliography}{99}
\bibitem {AhOn00} S. Ahlgren and K. Ono, A Gaussian hypergeometric series evaluation and Ap\'ery number
congruences, J. Reine Angew. Math. 518 (2000), 187--212.

\bibitem{AAR99} G.E. Andrews, R. Askey and R. Roy, Special Functions, Encyclopedia of Mathematics and its Applications 71, Cambridge University Press, Cambridge, 1999.

\bibitem{DFLST16} A. Deines, J.G. Fuselier, L. Long, H. Swisher and F.-T. Tu, Hypergeometric series, truncated hypergeometric series, and Gaussian hypergeometric functions, Directions in number theory, 125--159, Assoc. Women Math. Ser. 3, Springer, 2016.

\bibitem {GR} G. Gasper and M. Rahman, Basic Hypergeometric Series, 2nd ed., Encyclopedia of Mathematics and
its Applications 96, Cambridge University Press, Cambridge, 2004.

\bibitem{GuoLiu} V.J.W. Guo and J.-C. Liu, Some congruences related to a congruence of Van Hamme, Integral Transforms Spec. Funct. 31 (2020), 221--231.

\bibitem{GLS} V.J.W. Guo, J.-C. Liu and M.J. Schlosser, An extension of a supercongruence of Long and Ramakrishna, Proc. Amer. Math. Soc. 151 (2023), 1157--1166.

\bibitem{GS1} V.J.W. Guo and M.J. Schlosser, A family of $q$-hypergeometric congruences modulo the fourth power of a cyclotomic polynomial, Israel J. Math. 240 (2020), 821--835.

\bibitem{GS} V.J.W. Guo and M.J. Schlosser, Some $q$-supercongruences from transformation formulas for basic hypergeometric series, Constr. Approx. 53 (2021), 155--200.

\bibitem{Liu2017} J.-C. Liu, A $p$-adic supercongruence for truncated hypergeoemtric series ${}_7F_6$, Results Math. 72 (2017), 2057--2066.

\bibitem{Liu2021} J.-C. Liu, Supercongruences arising from transformations of hypergeometric series, J. Math. Anal. Appl. 497 (2021), Art. 124915.

\bibitem{LR} L. Long and R. Ramakrishna, Some supercongruences occurring in truncated hypergeometric series, Adv. Math. 290 (2016), 773--808.

\bibitem{MaoPan2020} G.-S. Mao and H. Pan, On the divisibility of some truncated hypergeometric series, Acta Arith. 195 (2020), 199--206.

\bibitem{MaoPan2022} G.-S. Mao and H. Pan, Congruences corresponding to hypergeometric identities I. ${}_2F_1$ transformations, J. Math. Anal. Appl. 505 (2022), Art. 125527.

\bibitem{Morita} Y. Morita, A $p$-adic analogue of the $\Gamma$-function, J. Fac. Sci. Univ. Tokyo Sect. IA Math. 22 (1975), no. 2, 255--266.

\bibitem{Mortenson} E. Mortenson, A $p$-adic supercongruence conjecture of Van Hamme, Proc. Amer. Math. Soc. 136 (2008), 4321--4328.

\bibitem{PTW} H. Pan, R. Tauraso and C. Wang, A local-global theorem for $p$-adic supercongruences, J. reine angew. Math. 790 (2022), 53--83.

\bibitem{Robert00} A.M. Robert, A Course in $p$-Adic Analysis, Graduate Texts in Mathematics, Vol. 198, Springer-Verlag, New York, 2000.

\bibitem{Sun2011} Z.-W. Sun, Super congruences and Euler numbers, Sci. China Math. 54 (2011), 2509--2535.

\bibitem{WangPan} C. Wang and H. Pan, Supercongruences concerning truncated hypergeometric series, Math. Z. 300 (2022), 161--177.

\bibitem{VanHamme} L. Van Hamme, Some conjectures concerning partial sums of generalized hypergeometric series, $p$-adic functional analysis (Nijmegen, 1996), Lecture Notes in Pure and Appl. Math. 192, Dekker, New York (1997), 223--236.
\end{thebibliography}
\end{document}